\documentclass[a4paper]{amsart}

\usepackage{amsfonts, latexsym, amssymb, graphicx,
multicol, mathrsfs, color, amscd, verbatim, paralist,
xspace, url, euscript, stmaryrd,  amsmath, amscd, enumitem,
bbold, multirow, tikz, mathtools, mathabx}
\usepackage[all,pdf,cmtip]{xy}

\usepackage[colorlinks, linkcolor=blue, citecolor=magenta, urlcolor=cyan]{hyperref}

\usepackage{mathabx}

\def\ra{\rightarrow}

\newtheorem{theorem}{THEOREM}[section]
\newtheorem{corollary}[theorem]{Corollary}
\newtheorem{proposition}[theorem]{Proposition}
\newtheorem{lemma}[theorem]{Lemma}

\theoremstyle{definition}

\theoremstyle{remark}
\newtheorem{remark}[theorem]{Remark}

\newcommand\CC{{\mathbb C}}
\newcommand\RR{{\mathbb R}}

\def\Re{\mathop{\rm Re}\nolimits}

\makeatletter
\def\blfootnote{\xdef\@thefnmark{}\@footnotetext}
\makeatother

\begin{document}

\title[Rigid Levi degenerate hypersurfaces with vanishing CR-curvature]{Rigid Levi degenerate hypersurfaces
\vspace{0.1cm}\\
with vanishing CR-curvature}
\author[Isaev]{Alexander Isaev}

\address{Mathematical Sciences Institute\\
Australian National University\\
Canberra, Acton, ACT 2601, Australia}
\email{alexander.isaev@anu.edu.au}

\maketitle

\thispagestyle{empty}

\pagestyle{myheadings}

\begin{abstract} 
We continue our study, initiated in article \cite{I3}, of a class of rigid hypersurfaces in $\CC^3$ that are 2-nondegenerate and uniformly Levi degenerate of rank 1, having zero CR-curvature. We drop the restrictive assumptions of \cite{I3} and give a complete description of the class. Surprisingly, the answer is expressed in terms of solutions of several well-known differential equations, in particular, the equation characterizing conformal metrics with constant negative curvature and a nonlinear $\bar\partial$-equation.
\end{abstract}

\section{Introduction}\label{intro}
\setcounter{equation}{0}

This paper is a continuation of articles \cite{I2, I3}, and we will extensively refer the reader to these papers in what follows. In particular, a brief review of CR-geometric concepts is given in \cite[Section 2]{I2}, and we will make use of those concepts without further reference.

We consider connected $C^{\infty}$-smooth real hypersurfaces in $\CC^n$, with $n\ge 2$. Specifically, we look at \emph{rigid hypersurfaces}, i.e., hypersurfaces given by an equation of the form
\begin{equation}
\Re z_n=F(z_1,\overline{z}_1,\dots, z_{n-1},\overline{z}_{n-1}),\label{rigideqgen}
\end{equation}
where $F$ is a smooth real-valued function defined on a domain in $\CC^{n-1}$. Rigid hypersurfaces have rigid CR-structures in the sense of \cite{BRT} and are invariant under the 1-parameter family of holomorphic transformations
\begin{equation}
(z_1,\dots,z_n)\mapsto (z_1,\dots,z_{n-1}, z_n+it),\,\, t\in\RR.\label{translatisymm}
\end{equation}

We only consider rigid hypersurfaces passing through the origin and will be interested in their germs at 0. Thus, we assume that $F$ in (\ref{rigideqgen}) is defined near the origin and $F(0)=0$, with the domain of $F$ being allowed to shrink if necessary.

We will utilize a natural notion of equivalence for germs of rigid hypersurfaces as introduced in \cite{I3}. Namely, two germs of rigid hypersurfaces at 0 are called \emph{rigidly equivalent} if there exists a map of the form
\begin{equation}
\begin{array}{l}
(z_1,\dots,z_n)\mapsto\Bigl(f_1(z_1,\dots,z_{n-1}),\dots, f_{n-1}(z_1,\dots,z_{n-1}),\\
\vspace{-0.3cm}\\
\hspace{6.1cm}a z_n+f_n(z_1,\dots,z_{n-1})\Bigr),\,\,a\in \RR^*,
\end{array}\label{rigideqmaps}
\end{equation}
nondegenerate at the origin, that transforms one hypersurface germ into the other, where $f_j$ is holomorphic near 0 with $f_j(0)=0$ for $j=1,\dots,n$. 

Our ultimate goal is 
\vspace{-0.7cm}\\

$$
\begin{array}{l}
\hspace{0.2cm}\hbox{$(*)$ to classify, up to rigid equivalence, the germs of rigid hypersurfaces that are}\\
\vspace{-0.3cm}\hspace{0.7cm}\hbox{\emph{CR-flat}, that is, have identically vanishing CR-curvature as explained below.}\\
\end{array}
$$
\vspace{-0.4cm}\\

The problem of describing CR-flat structures that possess certain symmetries (e.g., as in (\ref{translatisymm})) is a natural one but so far has only been addressed for Levi nondegenerate CR-hypersurfaces. In particular, under the assumption of CR-flatness, homogeneous strongly pseudoconvex CR-hypersurfaces have been studied (see \cite{BS}), and tube hypersurfaces in complex space have been extensively investigated and even fully classified for certain signatures of the (nondegenerate) Levi form (see \cite{I1} for a detailed exposition). Compared to the tube case, the case of rigid hypersurfaces is the next situation up in terms of complexity. Although not mentioned explicitly, problem $(*)$ was first looked at in article \cite{S} for real-analytic strongly pseudoconvex hypersurfaces in $\CC^2$. Even in this simplified setup, determining all rigid CR-flat hypersurfaces turned out to be highly nontrivial, with only a number of examples found in \cite{S}. A complete list (not entirely explicit but presented in an acceptable form) of the germs of real-analytic strongly pseudoconvex rigid CR-flat hypersurfaces in $\CC^2$ was only recently obtained in \cite{ES}. 

Our task is much more ambitious as we attempt to relax the Levi nondegeneracy assumption and do not assume real-analyticity. Specifically, in article \cite{I3} we initiated an investigation of problem $(*)$ for a class of Levi degenerate 2-nondegenerate rigid hypersurfaces in $\CC^3$ and obtained a partial description up to rigid equivalence. As part of our considerations, we analyzed CR-curvature for this class.

We will now briefly discuss the concept of CR-curvature. Generally, CR-curvature is defined in a situation when the CR-structures in question are reducible to absolute parallelisms with values in a Lie algebra ${\mathfrak g}$. Indeed, let ${\mathfrak C}$ be a class of CR-manifolds. Then the CR-structures in ${\mathfrak C}$ are said to \emph{reduce to ${\mathfrak g}$-valued absolute parallelisms} if to every $M\in{\mathfrak C}$ one can assign a fiber bundle ${\mathcal P}_M\ra M$ and an absolute parallelism $\omega_M$ on ${\mathcal P}_M$ such that for every $p\in {\mathcal P}_M$ the parallelism establishes an isomorphism between $T_p({\mathcal P}_M)$ and ${\mathfrak g}$, and for all $M_1,M_2\in{\mathfrak C}$ the following holds: 
\vspace{0.1cm}

\noindent (i) every CR-isomorphism $f:M_1\ra M_2$ can be lifted to a diffeomorphism\linebreak $F: {\mathcal P}_{M_{{}_1}}\ra{\mathcal P}_{M_{{}_2}}$ satisfying
\begin{equation}
F^{*}\omega_{M_{{}_2}}=\omega_{M_{{}_1}},\label{eq8}
\end{equation}
and 

\noindent (ii) any diffeomorphism $F: {\mathcal P}_{M_{{}_1}}\ra{\mathcal P}_{M_{{}_2}}$ satisfying (\ref{eq8}) 
is a bundle isomorphism that is a lift of a CR-isomorphism $f:M_1\ra M_2$.
\vspace{-0.3cm}\\

In this situation one considers the ${\mathfrak g}$-valued \emph{CR-curvature form}
$$
\Omega_M:=d\omega_M-\frac{1}{2}\left[\omega_M,\omega_M\right],\label{genformulacurvature}
$$
and the CR-flatness of $M$ is the condition of the identical vanishing of $\Omega_M$ on the bundle ${\mathcal P}_M$.
 
Reduction of CR-structures to absolute parallelisms was initiated by \'E. Cartan  who considered the case of 3-dimensional Levi nondegenerate CR-hyper\-surfaces (see \cite{C}). Since then there have been numerous developments under the Levi nondegeneracy assumption (see \cite[Section 1]{I2} for references). On the other hand, the first result for a reasonably large class of Levi degenerate manifolds is fairly recent. Namely, in article \cite{IZ} we looked at the class ${\mathfrak C}_{2,1}$ of connected 5-dimensional CR-hypersurfaces that are 2-nondegenerate and uniformly Levi degenerate of rank 1 and showed that the CR-structures in ${\mathfrak C}_{2,1}$ reduce to ${\mathfrak{so}}(3,2)$-valued parallelisms. Alternative constructions for ${\mathfrak C}_{2,1}$ were proposed in \cite{MS}, \cite{MP}, \cite{Poc}, \cite{FM} (see also \cite{Por}, \cite{PZ} for reduction in higher-dimensional cases). 

Everywhere in this paper we understand CR-curvature and CR-flatness for the class ${\mathfrak C}_{2,1}$ in the sense of article \cite{IZ}. One of the results of \cite{IZ} states that a manifold $M\in{\mathfrak C}_{2,1}$ is CR-flat if and only if in a neighborhood of its every point $M$ is CR-equivalent to an open subset of the tube hypersurface over the future light cone in $\RR^3$, i.e.,
$$
M_0:\quad(\Re z_1)^2+(\Re z_2)^2-(\Re z_3)^2=0,\,\ \Re z_3>0.\label{light}
$$
In fact, one can show that the germ of $M_0$ at its every point is CR-equivalent to the germ at the origin of the following rigid hypersurface:
$$
\tilde M_0:\quad\Re z_3=\frac{|z_1|^2}{1-|z_2|^2}+\frac{\bar z_2}{2(1-|z_2|^2)}z_1^2+\frac{z_2}{2(1-|z_2|^2)}\bar z_1^2\label{bestknownexample1}
$$
(see \cite{GM} and \cite[Proposition 4.16]{FK}). We thus deduce that for the class ${\mathfrak C}_{2,1}$ problem $(*)$ reduces to the determination, up to rigid equivalence, of all germs of rigid hypersurfaces in ${\mathfrak C}_{2,1}$ that are CR-equivalent to the germ of $\tilde M_0$. 

For real hypersurfaces in $\CC^3$ of the class ${\mathfrak C}_{2,1}$, the condition of local CR-equivalence to $\tilde M_0$ near the origin (i.e., the CR-flatness condition) can be expressed as the simultaneous vanishing of two CR-invariants, called $J$ and $W$, introduced by S.~Pocchiola in \cite{Poc} (cf.~\cite{FM}). These invariants are given explicitly in terms of a graphing function of the hypersurface. The general formulas in \cite{Poc} for $J$ and $W$ are rather lengthy and hard to work with. Luckily, they substantially simplify in the rigid case, and our arguments are based on those shorter formulas.

In article \cite{I3} we initiated the study of a class of solutions of the system
\begin{equation}
\left\{\begin{array}{l}
J=0,\\
\vspace{-0.3cm}\\
W=0
\end{array}\right.\label{condsjw}
\end{equation}  
for rigid hypersurfaces in $\CC^3$ lying in ${\mathfrak C}_{2,1}$. The class is given by conditions (\ref{s1111}) and (\ref{specclass}) stated in the next section. Quite unexpectedly, the study of even this special class of solutions turns out to be rather nontrivial and leads to interesting mathematics. In \cite{I3}, we only determined the solutions in the class satisfying certain additional assumptions as specified in Remark \ref{speccaseprevpaper}. In the present paper, we drop those assumptions and give a complete description, up to rigid equivalence, of all germs of rigid hypersurfaces in $\CC^3$ lying in ${\mathfrak C}_{2,1}$ whose graphing functions satisfy (\ref{s1111}) and (\ref{specclass}). This is our main result; it is stated in Theorem \ref{main}. One consequence of Theorem \ref{main} is the analyticity of solutions (see Corollary \ref{analyticity}). It would be curious to see whether analyticity holds true without any further constraints as discussed in Remark \ref{remanalyticity}.

 It is interesting to note that in the course of our analysis various classical differential equations kept appearing. Indeed, first of all, the complex homogeneous Monge-Amp\`ere equation (\ref{cmplxmongeampere}) arose simply because the Levi form of a hypersurface in ${\mathfrak C}_{2,1}$ is everywhere degenerate. Secondly, the condition $J=0$ mysteriously led to an analogue of the Monge equation (see (\ref{mplxmongeeq})). Thirdly, the equation describing conformal metrics with constant negative curvature unexpectedly came up (see (\ref{confmetrconsnegcurv})). Fourthly, we came across a nonlinear $\bar\partial$-equation, which is a special case of the much-studied equation $\partial u/\partial\bar z=Au+B\bar u+f$ (see (\ref{classnonlindbar})). The description given in Theorem \ref{main} is expressed in terms of solutions of (\ref{confmetrconsnegcurv}), (\ref{classnonlindbar}), as well as those of more elementary equations. We thus see that the geometry of CR-flat manifolds in ${\mathfrak C}_{2,1}$ is rather rich, has surprising connections with classical differential-geometric structures, and so deserves further investigation.

{\bf Acknowledgements.} This work was done while the author was visiting the Steklov Mathematical Institute in Moscow. I wish to thank Stefan Nemirovski, Alexandre Sukhov and Neil Trudinger for useful discussions.

\section{Results}\label{prelim}
\setcounter{equation}{0}

From now on, we will look at the germs of rigid hypersurfaces at the origin in $\CC^3$. They are given by equations of the form
$$
\Re z_3=F(z_1,\overline{z}_1,z_2,\overline{z}_2),\label{rigideq}
$$
where $F$ is a smooth real-valued function defined near the origin with $F(0)=0$. 

Consider the germ $M$ of a rigid hypersurface that is uniformly Levi degenerate of rank 1. Then the complex Hessian matrix of $F$ has rank 1 at every point, hence $F$ is a solution of the \emph{complex homogeneous Monge-Amp\`ere equation}
\begin{equation}
F_{1\bar1}F_{2\bar 2}-|F_{1\bar 2}|^2=0\label{cmplxmongeampere}
\end{equation}
(here and below subscripts $1$, $\bar 1$, $2$, $\bar 2$ indicate partial derivatives with respect to $z_1$, $\bar z_1$, $z_2$, $\bar z_2$, respectively). Clearly, we have either $F_{1\bar1}(0)\ne 0$ or $F_{2\bar2}(0)\ne 0$, so, up to rigid equivalence, we may additionally assume
\begin{equation}
\hbox{$F_{1\bar 1}>0$ everywhere.}\label{f11}
\end{equation}

Set
\begin{equation}
S:=\left(\frac{F_{1\bar 2}}{F_{1\bar 1}}\right)_1.\label{newdefs}
\end{equation}
The condition of 2-nondegeneracy of $M$ is then expressed as the nonvanishing of $S$ (see \cite{MP,Poc}). Thus, supposing that $M$ is 2-nondegenerate, we have
\begin{equation}
\hbox{$S\ne 0$ everywhere.}\label{cmplxsnonzero}
\end{equation}

Further, for real hypersurfaces in $\CC^3$ of the class ${\mathfrak C}_{2,1}$ the CR-flatness condition is equivalent to the simultaneous vanishing of two CR-invariants, called $J$ and $W$, introduced in \cite{Poc}. These invariants are given by explicit expressions in terms of a graphing function of the hypersurface. The formulas in \cite{Poc} for $J$ and $W$ are quite lengthy and hard to handle in general. However, these complicated formulas significantly simplify in the rigid case as the following proposition shows:

\begin{proposition}\label{complexexplinvprop}
For the germ $M$ of a rigid hypersurface in the class ${\mathfrak C}_{2,1}$ with $F$ satisfying {\rm (\ref{f11})}, we have 
\begin{equation}
\begin{array}{l}
\displaystyle J=\frac{5(S_1)^2}{18 S^2}\frac{F_{11\bar 1}}{F_{1\bar 1}}+\frac{1}{3}\frac{F_{11\bar 1}}{F_{1\bar 1}}\left(\frac{F_{11\bar 1}}{F_{1\bar 1}}\right)_1-\frac{S_1}{9 S}\frac{(F_{11\bar 1})^2}{(F_{1\bar 1})^2}+\frac{20(S_1)^3}{27 S^3}-\frac{5 S_1 S_{11}}{6 S^2}+\\
\vspace{-0.3cm}\\
\hspace{0.7cm}\displaystyle\frac{S_1}{6 S}\left(\frac{F_{11\bar 1}}{F_{1\bar 1}}\right)_1-\frac{S_{11}}{6 S}\frac{F_{11\bar 1}}{F_{1\bar 1}}-\frac{2}{27}\frac{(F_{11\bar 1})^3}{(F_{1\bar 1})^3}-\frac{1}{6}\left(\frac{F_{11\bar1}}{F_{1\bar 1}}\right)_{11}+\frac{S_{111}}{S_1},\\
\vspace{-0.1cm}\\
\displaystyle W=\frac{2\bar S_1}{3\bar S}+\frac{2S_1}{3S}+\frac{\bar S_{\bar 1}}{3\bar S^3}\left(\frac{F_{2\bar 1}}{F_{1\bar 1}}\bar S_1-\bar S_2\right)-\frac{1}{3\bar S^2}\left(\frac{F_{2\bar 1}}{F_{1\bar 1}}\bar S_{1\bar 1}-\bar S_{2\bar 1}\right).
\end{array}\label{cmplxexpljw}
\end{equation}
\end{proposition}

\noindent The proof of Proposition \ref{complexexplinvprop} goes by straightforward manipulation of formulas in \cite{Poc}, and we omit it.

Finding the germs of the graphs of all solutions of system (\ref{condsjw}) up to rigid equivalence is apparently very hard. In article \cite{I3} we made some initial steps towards this goal. Specifically, we discussed solutions having the property
\begin{equation}
S_1=0,\quad S_{\bar 1}=0.\label{s1111}
\end{equation}
Our motivation for introducing conditions (\ref{s1111}) comes from the tube case, where these conditions are equivalent to the equation $W=0$ (see \cite[Lemma 3.1]{I3}). At this stage, we do not know whether the same holds true in the rigid case as well, but it is clear from (\ref{cmplxexpljw}) that (\ref{s1111}) implies $W=0$.

Furthermore, as can be easily seen from (\ref{cmplxexpljw}), conditions (\ref{s1111}) lead to the following simplified expression for $J$:
\begin{equation}
\displaystyle J=\frac{1}{3}\frac{F_{11\bar 1}}{F_{1\bar 1}}\left(\frac{F_{11\bar 1}}{F_{1\bar 1}}\right)_1-\frac{2}{27}\frac{(F_{11\bar 1})^3}{(F_{1\bar 1})^3}-\frac{1}{6}\left(\frac{F_{11\bar1}}{F_{1\bar 1}}\right)_{11}.\label{reducedj}
\end{equation}
Formula (\ref{reducedj}) yields that under assumption (\ref{s1111}) the equation $J=0$ is equivalent to
\begin{equation}
9 F_{1111\bar 1}(F_{1\bar 1})^2-45 F_{111\bar 1}F_{11\bar 1} F_{1\bar 1}+40 (F_{11\bar 1})^3=0,\label{mplxmongeeq}
\end{equation}
which looks remarkably similar to the well-known \emph{Monge equation}. Recall that the classical single-variable Monge equation is
\begin{equation}
9f^{{\rm(V)}}(f'')^2-45f^{{\rm(IV)}}f'''f''+40(f''')^3=0\label{classmongeeq}
\end{equation}
and that it encodes all planar conics (see, e.g., \cite[pp.~51--52]{Lan}, \cite{Las}). In analogy with (\ref{classmongeeq}), we call (\ref{mplxmongeeq}) the \emph{complex Monge equation with respect to $z_1$}.

Thus, we arrive at a natural class of CR-flat rigid hypersurfaces in ${\mathfrak C}_{2,1}$ described by the system of partial differential equations
$$
\left\{\begin{array}{l}
\hbox{the complex Monge equation w.r.t. $z_1$ (\ref{mplxmongeeq})},\\
\vspace{-0.3cm}\\
\hbox{the complex Monge-Amp\`ere equation (\ref{cmplxmongeampere})},\\
\vspace{-0.3cm}\\
\hbox{equations (\ref{s1111})},
\end{array}\right.
$$
where conditions (\ref{f11}) and (\ref{cmplxsnonzero}) are satisfied. 

We will now recall \cite[Proposition 5.4]{I3}, which states that (\ref{mplxmongeeq}) can be integrated three times with respect to $z_1$. Note that an analogous fact holds for the classical Monge equation (see, e.g., \cite{I2}). 

\begin{proposition}\label{formfunctionf}
A function $F$ satisfying {\rm (\ref{f11})} is a solution of {\rm (\ref{mplxmongeeq})} if and only if
\begin{equation}
\begin{array}{l}
\displaystyle\frac{1}{(F_{1\bar 1})^{\frac{2}{3}}}=f(z_2,\bar z_2)|z_1|^4+g(z_2,\bar z_2)z_1^2\bar z_1+\overline{g(z_2,\bar z_2)}z_1\bar z_1^2+h(z_2,\bar z_2)|z_1|^2+\\
\vspace{-0.3cm}\\
\hspace{1.5cm}\displaystyle p(z_2,\bar z_2)z_1^2+\overline{p(z_2,\bar z_2)}\bar z_1^2+q(z_2,\bar z_2)z_1+\overline{q(z_2,\bar z_2)}\bar z_1+x(z_2,\bar z_2),
\end{array}\label{intthreetimes}
\end{equation}
where $f,g,h,p,q,x$ are smooth functions, with $f, h, x$ being real-valued.
\end{proposition}

In \cite{I3} we began investigating the simplest possible situation arising from Proposition \ref{formfunctionf} by assuming that in formula (\ref{intthreetimes}) one has
\begin{equation}
f=g=h=p=q=0.\label{specclass}
\end{equation}
This means that $F_{1\bar 1}=r(z_2,\bar z_2)$ or, equivalently,
\begin{equation}
F=r(z_2,\bar z_2)|z_1|^2+s(z_1,z_2,\bar z_2)+\overline{s(z_1,z_2,\bar z_2)}\label{specform1}
\end{equation}
for some functions $r$ and $s$ smooth near the origin, with $r>0$ everywhere and $\Re s(0)=0$. As explained in \cite{I3} and can be seen from what follows, exploring even this very special case is far from being trivial and leads to some interesting analysis. In \cite{I3} we introduced additional assumptions (shown in Remark \ref{speccaseprevpaper} below) and obtained a partial classification of the corresponding rigid hypersurfaces germs. In the present paper, we focus on functions of the form (\ref{specform1}) without any further constraints. Our goal is to determine $r$ and $s$ as explicitly as possible.

We will utilize the complex Monge-Amp\`ere equation (\ref{cmplxmongeampere}). Indeed, plugging (\ref{specform1}) in (\ref{cmplxmongeampere}) leads to
\begin{equation}
r(r_{2\overline{2}}|z_1|^2+s_{2\overline{2}}+\overline{s_{2\overline{2}}})-|r_2|^2|z_1|^2-|s_{1\overline{2}}|^2-r_2s_{1\overline{2}}z_1-r_{\overline 2}\overline{s_{1\overline{2}}}\bar z_1=0.\label{cmpxmapl}   
\end{equation}
We will now differentiate (\ref{cmpxmapl}). First, applying the operator $\partial^2/\partial z_1\partial\bar z_1$ to (\ref{cmpxmapl}) yields
\begin{equation}
rr_{2\bar 2}-|r_2|^2-|s_{11\bar 2}|^2=0.\label{eqforr}
\end{equation} 
Further, differentiating (\ref{eqforr}) with respect to $z_1$, we obtain
$$
s_{111\bar 2}\overline{s_{11\bar 2}}=0,
$$
which implies that $s_{111\bar 2}=0$, and therefore we have
$$
s_{\bar 2}=t_0(z_2,\bar z_2)z_1^2+u_0(z_2,\bar z_2)z_1+v_0(z_2,\bar z_2)
$$
for some smooth functions $t_0$, $u_0$, $v_0$ near the origin. Solving the $\bar\partial$-equations 
$$
t_{\bar 2}(z_2,\bar z_2)=t_0(z_2,\bar z_2),\quad u_{\bar 2}(z_2,\bar z_2)=u_0(z_2,\bar z_2),\quad v_{\bar 2}(z_2,\bar z_2)=v_0(z_2,\bar z_2)
$$
on a neighborhood of the origin, we obtain
$$
s=t(z_2,\bar z_2)z_1^2+u(z_2,\bar z_2)z_1+v(z_2,\bar z_2)+w(z_1,z_2),
$$
with $t$, $u$, $v$, $w$ being smooth functions defined near the origin and $\Re(v(0)+w(0))=0$. Since $w$ is in fact holomorphic and we study rigid hypersurfaces up to rigid equivalence, by absorbing $w$ into $z_3$ we may assume that $w=0$, so we have
\begin{equation}
s=t(z_2,\bar z_2)z_1^2+u(z_2,\bar z_2)z_1+v(z_2,\bar z_2),\label{formforsss}
\end{equation}
with $\Re v(0)=0$. Note that by condition (\ref{cmplxsnonzero}) we have $t_{\bar 2}\ne 0$ everywhere.

Next, applying the operator $\partial^2/\partial z_1^2$ to (\ref{cmpxmapl}), we obtain
$$
rs_{112\bar 2}-s_{111\bar 2}\overline{s_{1\bar 2}}-r_2s_{111\bar 2}z_1-2r_2s_{11\bar 2}=0,
$$
which, upon taking into account expression (\ref{formforsss}), leads to
\begin{equation}
rt_{2\bar 2}-2r_2t_{\bar 2}=0,\label{eqforttt}
\end{equation}
hence to
\begin{equation}
\frac{t_{2\bar 2}}{t_{\bar 2}}=2\frac{r_2}{r}.\label{integrable1}
\end{equation} 
Passing to logarithms and integrating (\ref{integrable1}) we arrive at the equation
\begin{equation}
t_{\bar 2}=w(\bar z_2)r^2,\label{formtbar}
\end{equation}
where $w(\bar z_2)$ is everywhere nonvanishing.  

\begin{lemma}\label{constant}
Up to rigid equivalence, one can assume that in {\rm (\ref{formtbar})} we have\linebreak $w\equiv1/4$.
\end{lemma}

\begin{proof} Let us perform the transformation
\begin{equation}
(z_1,z_2,z_3)\mapsto \left(\frac{z_1}{2\sqrt{\overline{w(\bar z_2)}}},z_2,z_3\right).\label{specmap}
\end{equation}
Clearly, (\ref{specmap}) is a map of the form (\ref{rigideqmaps}) for $n=3$ and preserves condition (\ref{f11}). Under this transformation the germ of the graph of $F$ transforms into the germ of the graph of the function
$$
\tilde F=\tilde r(z_2,\bar z_2)|z_1|^2+\tilde s(z_1,z_2,\bar z_2)+\overline{\tilde s(z_1,z_2,\bar z_2)},
$$
with 
$$
\tilde s=\tilde t(z_2,\bar z_2)z_1^2+\tilde u(z_2,\bar z_2)z_1+v(z_2,\bar z_2),
$$
where
\begin{equation}
\begin{array}{l}
\tilde r(z_2,\bar z_2):=4|w(\bar z_2)| r(z_2,\bar z_2),\\
\vspace{-0.3cm}\\
\tilde t(z_2,\bar z_2):=4\overline{w(\bar z_2)} t(z_2,\bar z_2),\\
\vspace{-0.3cm}\\
\tilde u(z_2,\bar z_2):=2\sqrt{\overline{w(\bar z_2)}\,}u(z_2,\bar z_2). 
\end{array}\label{rel1}
\end{equation}

Now, from (\ref{formtbar}) and (\ref{rel1}) we see
$$
\tilde t_{\bar 2}=\frac{1}{4}\tilde r^2
$$
as required.\end{proof}

By Lemma \ref{constant}, we may assume 
\begin{equation}
t_{\bar 2}=\frac{1}{4}r^2.\label{formtbar1}
\end{equation}
Plugging (\ref{formforsss}) and (\ref{formtbar1}) in (\ref{eqforr}), we obtain an equation for $r$:
\begin{equation}
rr_{2\bar 2}-|r_2|^2-\frac{1}{4}r^4=0.\label{eqforrr}
\end{equation}
It then follows that the function $R:=\ln r$ satisfies the equation
\begin{equation}
\Delta R=e^{2R}.\label{confmetrconsnegcurv}
\end{equation}
This shows that $r$ is a conformal metric of constant curvature -1 on a disk $U$ around the origin (see, e.g., \cite{KR}). Since $U$ is simply connected, by Liouville's theorem we have
\begin{equation}
r(z_2,\bar z_2)=\frac{2|\rho'(z_2)|}{1-|\rho(z_2)|^2},\label{explformforr}
\end{equation}
where $\rho$ is a holomorphic function on $U$ with nowhere vanishing derivative and values in the unit disk. In particular, $r$ is real-analytic.

\begin{remark}\label{speccaseprevpaper}
Formula (\ref{explformforr}) is the most explicit expression for $r$ that one can hope to obtain without making further assumptions. We note that for $r$ depending only on either $\Re z_2$ or $|z_2|^2$, it is possible to derive more precise formulas as shown in \cite[Theorems 5.5 and 5.9]{I3}. The dependence of $r$ on either $\Re z_2$ or $|z_2|^2$, as well as the conditions $u=0$, $v=0$, were the additional assumptions imposed in \cite{I3}, and it is under these assumptions the partial classification of the corresponding rigid hypersurfaces was produced.
\end{remark}

Now, for $r$ found in (\ref{explformforr}), we may solve $\bar\partial$-equation (\ref{formtbar1}) as follows:
$$ 
t=\int \frac{r^2}{4}d\bar z_2+w(z_2),
$$
where the integral in the right-hand side stands for the term-by-term integration, with respect to $\bar z_2$, of the power series in $z_2$, $\bar z_2$ representing $r^2/4$ near the origin.  
As $w$ is holomorphic, by absorbing $z_1^2w(z_2)$ into $z_3$ we may assume that $w=0$, so we have 
\begin{equation} 
t=\int \frac{r^2}{4}d\bar z_2.\label{formfort}
\end{equation}
In particular, $t$ is real-analytic. 

Now that we have found $r$ and $t$, in order to determine the function $F$ it remains to compute $u$ and $\Re v$ in formula (\ref{formforsss}). Plugging (\ref{formforsss}) into (\ref{cmpxmapl}), collecting the terms linear in $z_1$, and utilizing (\ref{formtbar1}), we obtain
\begin{equation}
ru_{2\bar 2}-2t_{\bar 2}\bar u_2-r_2u_{\bar 2}=0,\label{eqforuuu}
\end{equation}  
which yields
$$
\left(\frac{u_{\bar 2}}{r}\right)_2=\frac{1}{2}\bar u_2.
$$
Integrating we get
$$
u_{\bar 2}=r\left(\frac{1}{2}\bar u+w(\bar z_2)\right).
$$
By adding to $z_3$ the term $4z_1\overline{w(\bar z_2)}$ we may assume that $w=0$, thus $u$ satisfies 
\begin{equation}
u_{\bar 2}=\frac{r}{2}\bar u.\label{classnonlindbar}
\end{equation}
Every solution of (\ref{classnonlindbar}) is real-analytic (see \cite[pp.~143--144, Section 3.4]{V} and references therein, as well as \cite[Section 6.6]{M}).

Note that by the Cauchy-Pompeiu formula, on any disk $\widetilde U$ around the origin relatively compact in $U$, the function $u$ solves the integral equation
\begin{equation}
u(z_2,\bar z_2)=-\frac{1}{2\pi i}\int_{\widetilde U}\frac{r(\zeta,\bar\zeta)\overline{u(\zeta,\bar\zeta)}}{2(\zeta-z_2)}d\bar\zeta\wedge d\zeta+w(z_2),\label{intgeq}
\end{equation}
where the holomorphic function $w$ is given by
$$
w(z_2)=\frac{1}{2\pi i}\int_{\partial\widetilde U}\frac{u(\zeta,\bar \zeta)}{\zeta-z_2} d\zeta.
$$
Thus, $u$ is a solution of an integral equation of the form (\ref{intgeq}) for a suitable function $w$. Regarding the existence, regularity and uniqueness of solutions of such integral equations we refer the reader to \cite[Chapter III]{V} and, in several variables, to \cite[pp.~436--438]{NW}. Other representations of $u$ can be found in the references provided in \cite[pp.~143--144, Section 3.4]{V}.  

Now, we plug (\ref{formforsss}) into (\ref{cmpxmapl}) and collect the terms independent of $z_1$. Taking into account equation (\ref{classnonlindbar}), we obtain
\begin{equation}
(\Re v)_{2\bar 2}=\frac{r}{8}|u|^2.\label{eqforvvv}
\end{equation}
As $r$ and $u$ are real-analytic, we see
$$
(\Re v)_2=\int \frac{r}{8}|u|^2d\bar z_2+w_0(z_2),
$$
where, as before, the integral denotes the result of the term-by-term integration of the corresponding power series with respect to $\bar z_2$. Hence 
$$
\Re v=\int\left(\int \frac{r}{8}|u|^2d\bar z_2\right)dz_2+w(z_2)+\overline{w(z_2)},
$$
where the outer integral in the right-hand side stands for the term-by-term integration, with respect to $z_2$, of the power series in $z_2$, $\bar z_2$ representing the inner integral near the origin. Since $w$ is holomorphic, by absorbing $w(z_2)$ into $z_3$ we may assume that $w=0$, so we have
\begin{equation}
\Re v=\int\left(\int \frac{r}{8}|u|^2d\bar z_2\right)dz_2.\label{formforrev}
\end{equation}
In particular, $\Re v$ is real-analytic.

We arrive at the following result:

\begin{theorem}\label{main}
The germ of a rigid hypersurface in $\CC^3$ of the class ${\mathfrak C}_{2,1}$ with graphing function satisfying conditions {\rm (\ref{s1111}) and (\ref{specclass})} is rigidly equivalent to the germ of a rigid hypersurface with graphing function of the form {\rm (\ref{specform1})}, where $s$ is given by {\rm (\ref{formforsss})} and the functions $r$, $t$, $u$, $\Re v$ are determined from {\rm (\ref{explformforr}), (\ref{formfort}), (\ref{classnonlindbar}), (\ref{formforrev})}, respectively.

Conversely, the germ of a rigid hypersurface with graphing function of the form {\rm (\ref{specform1})}, where $s$ is given by {\rm (\ref{formforsss})} and the functions $r$, $t$, $u$, $\Re v$ are determined from {\rm (\ref{explformforr}), (\ref{formfort}), (\ref{classnonlindbar}), (\ref{formforrev})}, respectively, is of the class ${\mathfrak C}_{2,1}$ and the graphing function satisfies conditions {\rm (\ref{s1111}), (\ref{specclass})}.  
\end{theorem}

\begin{proof} We only need to prove the converse implication. First, an easy calculation shows that every function that arises in the right-hand side of (\ref{explformforr}) is a solution of (\ref{eqforrr}). Next, by plugging (\ref{formforsss}) into (\ref{cmpxmapl}) we see that conditions (\ref{eqforrr}), (\ref{formfort}), (\ref{classnonlindbar}), (\ref{formforrev}) guarantee that equation (\ref{cmpxmapl}) is satisfied. The latter equation is the complex Monge-Amp\`ere equation in the case at hand. Furthermore, $F_{1\bar 1}=r$, which is positive by (\ref{explformforr}), and therefore (\ref{f11}) holds. We have thus shown that the graph of $F$ has Levi form of rank 1 everywhere. 

Next, from (\ref{newdefs}), (\ref{formforsss}) we have
\begin{equation}
S=\frac{2t_{\bar 2}}{r},\label{formforSSS}
\end{equation}
and (\ref{formfort}) yields
$$
t_{\bar 2}=\frac{r^2}{4}\ne 0\,\,\hbox{everywhere}.
$$
Therefore, by (\ref{cmplxsnonzero}) the graph of $F$ is 2-nondege\-nerate, hence lies in the class ${\mathfrak C}_{2,1}$. It also follows from (\ref{formforSSS}) that conditions (\ref{s1111}) are satisfied. Finally, (\ref{specclass}) trivially holds, which concludes the proof.\end{proof}

We have an immediate consequence:

\begin{corollary}\label{analyticity}
The germ of a rigid hypersurface in $\CC^3$ of the class ${\mathfrak C}_{2,1}$ satisfying conditions {\rm (\ref{s1111}) and (\ref{specclass})} is real-analytic.
\end{corollary}

\begin{remark}\label{remanalyticity}
It would be interesting to see whether the real-analyticity result of Corollary \ref{analyticity} holds true for all CR-flat rigid hypersurfaces in $\CC^3$ of the class 
${\mathfrak C}_{2,1}$, regardless of assumptions (\ref{s1111}), (\ref{specclass}). Note that in general a CR-flat manifold does not have to be real-analytic. For example, it is easy to construct, for any $n\ge 2$, an example of a hypersurface in $\CC^n$ that is CR-equivalent to an open subset of $S^{2n-1}\subset\CC^n$ and $C^{\infty}$-smooth but not real-analytic (see, e.g., \cite[Remark 3.3]{I1}). In the Levi nondegenerate case the tubularity condition forces real-analyticity (see \cite[Proposition 3.1]{I1}) but it is unknown whether the rigidity condition is powerful enough for that. We stress that the work \cite{S, ES} for rigid Levi nondegenerate hypersurfaces in $\CC^2$ assumes real-analyticity as the techniques of the proofs rely on power series representations.
\end{remark}


\begin{thebibliography}{ABC}

\bibitem[BRT]{BRT} Baouendi, M. S., Rothschild, L. P. and Treves, F., CR structures with group action and extendability of CR functions, {\it Invent. Math.} {\bf 82} (1985), 359--396.

\bibitem[BS]{BS} Burns, D. and Shnider, S., Spherical hypersurfaces in complex manifolds, {\it Invent. Math.} {\bf 33} (1976), 223--246.
  
\bibitem[C]{C} Cartan, \' E., Sur la g\'eometrie pseudo-conforme des hypersurfaces de l'espace de deux variables complexes: I, {\it Ann. Math. Pura Appl.} {\bf 11} (1933), 17--90; II, {\it Ann. Scuola
Norm. Sup. Pisa} {\bf 1} (1932), 333--354.

\bibitem[ES]{ES} Ezhov, V. and Schmalz, G., Explicit description of spherical rigid hypersurfaces in $\CC^2$, {\it Complex Anal. Syner.} {\bf 1}:2 (2015), DOI: 10.1186/2197-120X-1-2.

\bibitem[FK]{FK} Fels, G. and Kaup, W., CR-manifolds of dimension 5: A Lie algebra approach, {\it J. reine angew. Math.} {\bf 604} (2007), 47--71.

\bibitem[FM]{FM} Foo, W. G. and Merker, J., Differential $\{e\}$-structures for equivalences of 2-nondegenerate Levi rank 1 hypersurfaces $M^5\subset\CC^3$, preprint, available from https://arxiv.org/abs/1901.02028.

\bibitem[GM]{GM} Gaussier, H. and Merker, J., A new example of a uniformly Levi degenerate hypersurface in $\CC^3$, {\it Ark. Mat.} {\bf 41} (2003), 85--94; erratum {\it Ark. Mat.} {\bf 45} (2007), 269--271.


\bibitem[I1]{I1} Isaev, A. V., {\it Spherical Tube Hypersurfaces}, Lecture Notes in Mathematics {\bf 2020}, Springer, New York, 2011.

\bibitem[I2]{I2} Isaev, A. V., Affine rigidity of Levi degenerate tube hypersurfaces, {\it J. Differential Geom.} {\bf 104} (2016), 111--141.


\bibitem[I3]{I3} Isaev, A. V., Zero CR-curvature equations for Levi degenerate hypersurfaces via Pocchiola's invariants, preprint, available from https://arxiv.org/abs/1809.03029.

\bibitem[IZ]{IZ} Isaev, A. and Zaitsev, D., Reduction of five-dimensional uniformly Levi degenerate CR structures to absolute parallelisms, {\it J. Geom. Anal.} {\bf 23} (2013), 1571--1605.

\bibitem[KR]{KR} Kraus, D. and Roth, O., Conformal metrics, in: \emph{Topics in Modern Function Theory}, Ramanujan Math. Soc. Lect. Notes Ser. 19, Ramanujan Math. Soc., Mysore, 2013, pp. 41--83.

\bibitem[Lan]{Lan} Landsberg, J. M., Differential-geometric characterizations of complete intersections, {\it J. Differential Geom.} {\bf 44} (1996), 32--73.

\bibitem[Las]{Las} Lasley, J. W., Jr., On Monge's differential equation, {\it Amer. Math. Monthly} {\bf 43} (1936), 284--286.

\bibitem[MS]{MS} Medori, C. and Spiro, A., The equivalence problem for five-dimensional Levi degenerate CR manifolds, {\it Int. Math. Res. Not. {\rm (}IMRN{\rm )}} (2014), 5602--5647.

\bibitem[MP]{MP} Merker, J. and Pocchiola, S., Explicit absolute parallelism for 2-nondegenerate real hypersurfaces $M^5\subset\CC^3$ of constant Levi rank 1, to appear in {\it J. Geom. Analysis}, published online, DOI: 10.1007/s12220-018-9988-3.

\bibitem[M]{M} Morrey, C. B., {\it Multiple Integrals in the Calculus of Variations}, Springer, Berlin, 2008.

\bibitem[NW]{NW} Nijenhuis, A. and Woolf, W. B., Some integration problems in almost-complex and complex manifolds, {\it Ann. Math.} {\bf 77} (1963), 424--489. 

\bibitem[Poc]{Poc} Pocchiola, S., Explicit absolute parallelism for 2-nondegenerate real hypersurfaces $M^5\subset\CC^3$ of constant Levi rank 1, preprint, available from https://arxiv.org/abs/1312.6400.

\bibitem[Por]{Por} Porter, C., The local equivalence problem for 7-dimensional 2-nondegenerate CR manifolds whose cubic form is of conformal unitary type, preprint, available from http://arxiv.org/abs/1511.04019.

\bibitem[PZ]{PZ} Porter, C. and Zelenko, I., Absolute parallelism for 2-nondegenerate CR structures via bigraded Tanaka prolongation, preprint, available from https://arxiv.org/abs/1704.03999.

\bibitem[S]{S} Stanton, N., A normal form for rigid hypersurfaces in $\CC^2$, {\it Amer. J. Math.} {\bf 113} (1991), 877--910.


\bibitem[V]{V} Vekua, I. N., {\it Generalized Analytic Functions}, Pergamon Press, New York, 1962.

\end{thebibliography}
\end{document}